\documentclass{amsart}[12 pt]
\usepackage{latexsym}
\usepackage{amsmath,amssymb, amsfonts}
\usepackage{verbatim}
\usepackage{graphicx}

\usepackage{ifpdf}
\ifpdf
  \DeclareGraphicsExtensions{.pdf,.png,.jpg}
\else
  \DeclareGraphicsExtensions{.eps}
\fi

\newtheorem{corollary}{Corollary}[section]
\newtheorem{lemma}{Lemma}[section]
\newtheorem{remark}{Remark}[section]

\numberwithin{equation}{section}

\begin{document}

\title[Corrigendum to ``Weak approximations for Wiener functionals'' ]{Corrigendum to ``Weak Approximations for Wiener Functionals'' [Ann. Appl. Probab. (2013)~\textbf{23}, 4, 1660-1691]}


\author{Dorival Le\~ao}

\address{Departamento de Matem\'atica Aplicada e Estat\'istica, Universidade de S\~ao Paulo, 13560-970, S\~ao Carlos - S\~ao Paulo, Brazil}\email{leao@icmc.usp.br}

\author{Alberto Ohashi}

\address{Departamento de Matem\'atica, Universidade Federal da Para\'iba, 13560-970, Jo\~ao Pessoa - Para\'iba, Brazil}\email{alberto.ohashi@pq.cnpq.br; ohashi@mat.ufpb.br}

\thanks{}
\date{\today}


\maketitle

Unfortunately, the proofs of Theorem 3.1 and Corollary 4.1 in our paper~\cite{LEAO_OHASHI2013} are incomplete. The reason is a wrong statement in Remark 2.2 in \cite{LEAO_OHASHI2013}. It is \textit{not} true that $\delta^k X$ is a square-integrable martingale for every square-integrable Brownian martingale $X$ (see Corollary \ref{AKMAR} below). As a consequence, the proofs of Lemmas 3.4 and 3.5 in \cite{LEAO_OHASHI2013} only cover the case when $\delta^kW$ is a pure jump martingale. Hence, the arguments given in the proofs of Theorem 3.1 and Corollary 4.1 have to be modified. The hypotheses and statements of Theorem 3.1 and Corollary 4.1 in \cite{LEAO_OHASHI2013} remain unchanged. In this note, we provide the correct proof of these results.

\section{Martingale Property of $\delta^kX$}
In the sequel, the notation of \cite{LEAO_OHASHI2013} is employed. Let $\textbf{B}^2(\mathbb{F})$ be the space of c\`adl\`ag $\mathbb{F}$-adapted processes on $\mathbb{R}_+$ such that $\mathbb{E}\sup_{t\ge 0}|X_t|^2< \infty$ and let $\textbf{H}^2(\mathbb{F})$ be the subspace of martingales $X\in \textbf{B}^2(\mathbb{F})$ such that $X_0=0$. For simplicity, we write $\textbf{B}^2$ and $\textbf{H}^2$ when no confusion arises about the filtration. Throughout this note, we fix a positive time $0< T < \infty$. In \cite{LEAO_OHASHI2013}, we have introduced the following operator acting on $\textbf{B}^2$,

$$\delta^kX_t = \sum_{n=0}^\infty\mathbb{E}\big[X_{T^k_n}|\mathcal{G}^k_n\big] 1\!\!1_{\{T^k_n \le t < T^k_{n+1}\}}; 0\le t\le T.$$

At first, let us clarify the $\mathbb{F}^k$-martingale property of $\delta^kX$ when $X\in\textbf{H}^2$. At first, we recall that $T^k_1 < \infty$~a.s so that the strong Markov property yields that $T^k_n < \infty~a.s$ for every $k,n\ge 1$. Let us denote $\Delta T^k_n:=T^k_n-T^k_{n-1}; n\ge 1$. By the very definition, $\mathcal{G}^k_{n} = \sigma(T^k_1,\ldots , T^k_n, \sigma^k_1, \ldots, \sigma^k_n) = \sigma(T^k_1, \Delta T^k_2, \ldots, \Delta T^k_n, \sigma^k_1, \ldots, \sigma^k_n); n\ge 1$. In particular, $\mathcal{G}^k_{1-}:=\mathcal{F}^k_{T^k_1-} = \sigma(T^k_1)$ and
$$\mathcal{G}^k_{n-} := \mathcal{F}^k_{T^k_n-}=\sigma(T^k_1, \Delta T^k_2, \ldots,\Delta T^k_{n-1}, \Delta T^k_n, \sigma^k_1, \ldots, \sigma^k_{n-1}); n\ge 2.$$


\begin{lemma}\label{martlemma}
Let $\{\xi^k_n; n\ge 1\}$ be a sequence of integrable random variables such that $\xi^k_n$ is $\mathcal{G}^k_n$-measurable for each $n\ge 1$. A pure jump process of the form $\sum_{n=1}^\infty \xi^k_n 1\!\!1_{\{T^k_n\le t\} }$ is an $\mathbb{F}^k$-martingale if, and only if, $\mathbb{E}[\xi^k_n|\mathcal{G}^k_{n-}]=0$ a.s for every $n\ge 1$.
\end{lemma}
\begin{proof}
This is an immediate consequence of Prop. I.1 in \cite{lejan} and the linearity of the space of martingales.
\end{proof}
By applying Lemma \ref{martlemma} to the process $\delta^kX$ for $X\in \textbf{H}^2$, we get the following characterization.
\begin{corollary}\label{cormart}
Let $X\in \textbf{H}^2$ be a Brownian  martingale and $X_\infty:=\lim_{t\rightarrow \infty}X_t~a.s$. The process $\delta^kX$ is an $\mathbb{F}^k$-martingale if, only if,
$\mathbb{E}[X_\infty|\mathcal{G}^k_{_n-}] -\mathbb{E}[X_\infty|\mathcal{G}^k_{n-1}]=0~a.s; n\ge 1$.
\end{corollary}



Since $\{T^k_n; n\ge 1\}$ is a sequence of totally inaccessible $\mathbb{F}^k$-stopping times, then Corollary \ref{cormart} implies that Remark 2.2 in \cite{LEAO_OHASHI2013} is \textbf{false}, i.e., $\delta^kX$ is \textbf{not} an $\mathbb{F}^k$-martingale for every Brownian martingale. However, we only need the martingale property of $\delta^kX$ in case $X$ is the Brownian motion.

\begin{corollary}\label{AKMAR}
The process $A^k$ is a square-integrable $\mathbb{F}^k$-martingale and it has the representation $A^k_t=\delta^kB_t =\mathbb{E}[B_T|\mathcal{F}^k_t]; 0\le t\le T$.
\end{corollary}
\begin{proof}
Let us define $C^k_t:=\max\{n\ge 0; T^k_n \le t\};t\ge 0$. We observe $C^k$ is independent from $B_{T^k_1}$. Indeed, we shall write $C^k_t= 2^{2k}[A^k,A^k]_t; t\ge 0$ and by definition $B_{T^k_1}$ is a Bernoulli variable of the form $B_{T^k_1}= 2^{-k}~\text{if}~\Delta A^k_{T^k_1}>0$ and $B_{T^k_1}=-2^{-k}~\text{if}~\Delta A^k_{T^k_1}< 0$. Then, we clearly see $B_{T^k_1}$ is independent from

$$C^k_t=2^{2k}\sum_{n=1}^\infty|\Delta A^k_{T^k_n}|^2 1\!\!1_{\{T^k_n\le t\}}; t\ge 0.$$
In one hand, for every $t\ge 0$, we have $\{C^k_t= n\} = \{T^k_{n} \le t < T^k_{n+1}\}; n\ge 0.$ On the other hand, $\{T^k_1\le t\} = \cup_{j=1}^{+\infty}\{C^k_t = j\}$ for every $t\ge 0$. In other words, the $\pi$-system $\big\{\{T^k_1\le t\}; t\ge 0\big\}$ which generates $\sigma(T^k_1)$ is independent from $B_{T^k_1}$. Therefore, $B_{T^k_1}$ is independent from $\sigma(T^k_1)$. By applying the strong Markov property, we then have $\mathbb{E}[B_{T^k_n} - B_{T^k_{n-1}}|\mathcal{G}^k_{n-}]=\mathbb{E}[B_{T^k_n} - B_{T^k_{n-1}}]=0; n\ge 1$, and from Lemma \ref{martlemma}, we conclude that $A^k$ is an $\mathbb{F}^k$-martingale. Representation $A^k_\cdot=\mathbb{E}[B_T|\mathcal{F}_\cdot]$ is just a consequence of the martingale property of the Brownian motion and the tower property.

\end{proof}






\section{Compactness of purely discontinuous $\mathbb{F}^k$-martingales}
For a given $X\in \textbf{H}^2$, let $\delta ^kX = M^{k,X} + N^{k,X}$ be the special $\mathbb{F}^k$-special semimartingale decomposition of $\delta^kX$, where $M^{k,X}$ is the martingale component of $\delta^kX$. Let $\textbf{H}^2(\mathbb{F}^k)$ be the space of all square-integrable $\mathbb{F}^k$-martingales starting at zero. From \cite{jacod}, we know that any square-integrable $\mathbb{F}^k$-martingale has bounded variation paths and it is purely discontinuous whose jumps are exhausted by $\cup_{n\ge 1} [[T^k_n,T^k_n]]$. In this case, any $Y^k\in \textbf{H}^2(\mathbb{F}^k)$ can be uniquely written as

\begin{equation}\label{smdec}
Y^k_t = Y^{k,pj}_t - N^{k,Y^k}_t; t\ge 0,
\end{equation}
where $N^{k,Y^k}$ is an $\mathbb{F}^k$-predictable continuous bounded variation process, $Y^{k,pj}_t:=\sum_{0< s\le t}\Delta Y^k_s; t\ge 0$ and $Y^{k,pj}_0 = N^{k,Y^k}_0=0$. From Th. 1 and 2 in \cite{jacod}, we can always write

$$Y^{k,pj}_t = \sum_{n=1}^\infty \Delta Y^k_{T^k_n} 1\!\!1_{\{T^k_n\le t\}}; t\ge 0.$$

As explained in Corollary \ref{cormart}, $\delta ^kW$ may not be an $\mathbb{F}^k$-martingale for a generic $W\in \textbf{H}^2$. Then, Lemma 3.4 and 3.5 in \cite{LEAO_OHASHI2013} may not be true in full generality, i.e., for every $W\in \textbf{H}^2$. However, if $W=B$ is the Brownian motion, then both lemmas are correct because $A^k=\delta^kB$ is a pure jump martingale as demonstrated in Corollary \ref{AKMAR}. In this case, the application of these lemmas based on $A^k$ to Proposition 3.2 in \cite{LEAO_OHASHI2013} is correct. However, in order to prove Theorem 3.1 and Corollary 4.1 in \cite{LEAO_OHASHI2013}, we still need to prove the following lemma.

\begin{lemma}\label{mainlemma}
Let $\delta ^kX = M^{k,X} + N^{k,X}$ be the canonical semimartingale decomposition for a Brownian martingale $X\in \textbf{H}^2$. Then,

\begin{equation}\label{convM}
M^{k,X}\rightarrow X
\end{equation}
weakly in $\textbf{B}^2$ over $[0,T]$ as $k\rightarrow \infty$. Moreover, $\langle X, B\rangle^\delta=[X,B]~\forall X\in \textbf{H}^2$.
\end{lemma}
Before proving the above lemma, we need some auxiliary results. At first, we observe that Prop 3.1 in \cite{LEAO_OHASHI2013} holds for any sequence $\{Y^k; k\ge 1\}$ of the form (\ref{smdec}).

\begin{lemma}\label{energylemma}
Let $\{Y^k; k\ge 1\}$ be a sequence of square-integrable martingales $Y^k\in \textbf{H}^2(\mathbb{F}^k); k\ge 1$. If  $\sup_{k\ge 1}\mathbb{E}[Y^k,Y^k]_T < \infty$, then $\{Y^k; k\ge 1\}$ is $\textbf{B}^2$-weakly relatively sequentially compact where all limit points are $\mathbb{F}$-square-integrable martingales over $[0,T]$.
\end{lemma}
\begin{proof} By denoting $Z^{k}_t: = \mathbb{E}[Y^{k}_T|\mathcal{F}_t];~0\le t\le T$, we can apply exactly the same arguments given in the proof of Proposition 3.1 in \cite{LEAO_OHASHI2013} to show that both $\{Z^k;k\ge 1\}$ and $\{Y^k;k\ge 1\}$ are $\textbf{B}^2$-weakly relatively compact and all limit points are $\mathbb{F}$-square-integrable martingales over $[0,T]$.

\end{proof}
Lemma 3.4 in \cite{LEAO_OHASHI2013} holds if $\delta^kW$ is a pure jump martingale. Then, we have the following result.
\begin{lemma}\label{lfund1}
Let $H_\cdot = \mathbb{E}[ 1\!\!1_{G}|\mathcal{F}_\cdot]$ and $H^k_\cdot =\mathbb{E}[ 1\!\!1_{ G}|\mathcal{F}^k_\cdot]$ be positive and uniformly integrable martingales w.r.t filtrations $\mathbb{F}$ and $\mathbb{F}^k$, respectively, where $G\in \mathcal{F}_T$. Then,

$$\Bigg\|\int_0^\cdot H_sdB_s -  \oint_0^\cdot H^k_sdA^k_s\Bigg\|_{\textbf{B}^2}\rightarrow 0\quad \text{as}~k\rightarrow \infty$$
over $[0,T]$.
\end{lemma}
\begin{proof}
Since $A^k$ is a pure jump martingale and $\mathbb{E}\sup_{0\le t\le T}|B_t|^p< \infty$ for every $p> 2$, then we shall apply Lemma 3.4 in \cite{LEAO_OHASHI2013} to conclude the proof.
\end{proof}
We observe that Lemma 3.5 in \cite{LEAO_OHASHI2013} holds for $A^k$ and a generic $Y^k$of the form (\ref{smdec}) as follows.
\begin{lemma}\label{lfund2}
Let $\{Y^k;k\ge 1\}$ be a sequence satisfying the assumption in Lemma \ref{energylemma}. Let $\{Y^{k_i}; i\ge 1\}$ be a $\textbf{B}^2$-weakly convergent subsequence such that $\lim_{i\rightarrow \infty}Y^{k_i}=Z$, where $Z\in \textbf{H}^2$. Then,

\begin{equation}\label{conq}
\lim_{k\rightarrow \infty}[Y^{k_i},A^{k_i}]_t = [Z,B]_t\quad\text{weakly in}~L^1(\mathbb{P})
\end{equation}
for every $t\in [0,T]$.
\end{lemma}
\begin{proof}
In the notation of the proof of Lemma 3.5 in~\cite{LEAO_OHASHI2013}, we observe that since for every BMO $\mathbb{F}$-martingale $U$, we have $\lim_{k\rightarrow \infty}[Z^{k,X}, U]_t=[Z,U]_t$ weakly in $L^1(\mathbb{P})$ for every $t\in [0,T]$, then we shall take $W=B$. We replace the martingale component $M^{k,X}$ defined by (2.10) in \cite{LEAO_OHASHI2013} by $Y^k$ in (\ref{smdec}). Then, by observing $\Delta Y^k = \Delta Y^{k,pj}$ and applying Lemmas \ref{energylemma} and \ref{lfund1}, the proof of Lemma 3.5 in \cite{LEAO_OHASHI2013} works perfectly for the pure-jump sequence $\{Y^{k,pj}; k\ge 1\}$ associated to the martingale components $\{Y^k; k\ge 1\}$.
\end{proof}
In the sequel, we fix $X\in \textbf{H}^2$ and write $X^k_t:=\mathbb{E}[X_T|\mathcal{F}^k_t]; t\ge 0$. Let, $X^k_t = X^{k,pj}_t - N^{k,X^k}_t;t\ge 0,$ be the $\mathbb{F}^k$-special semimartingale decomposition given in (\ref{smdec}). Let $\delta ^kX = M^{k,X} +N^{k,X}$ be the special semimartingale decomposition given by (2.10) in \cite{LEAO_OHASHI2013}. Since $X\in \textbf{H}^2$ and $\mathbb{F}^k\subset \mathbb{F}$ for every $k\ge 1$, then

$$\mathbb{E}[X_T|\mathcal{F}^k_t] = \mathbb{E}\big[ \mathbb{E}[X_\infty|\mathcal{F}_T]|\mathcal{F}^k_t\big]=\mathbb{E}[X_\infty|\mathcal{F}^k_t]; 0\le t\le T$$ so that $\mathbb{E}[X_T|\mathcal{G}^k_n] = \mathbb{E}[X_\infty|\mathcal{G}^k_n] = \mathbb{E}\big[ \mathbb{E}[X_\infty|\mathcal{F}_{T^k_n}]|\mathcal{G}^k_{n} \big] = \mathbb{E}[X_{T^k_n}|\mathcal{G}^k_n]$ on $\{T^k_n \le T\}$. In other words,

\begin{equation}\label{ident1}
X^k_{T^k_n} = \delta^kX_{T^k_n}~\text{on}~\{T^k_n\le T\}; k\ge 1.
\end{equation}
Let us denote $W^k:= X^k - M^{k,X}; k\ge 1$. Since $W^k$ is a purely discontinuous martingale, then it has a decomposition of the form (\ref{smdec}).

\begin{lemma} \label{carc_W}
The sequence $\{W^k ; k \geq 1 \}$ satisfies $\sup_{k\ge 1}\mathbb{E}[W^k,W^k]_T <\infty$ and

$$\Delta W^k_{T^k_n}  1\!\!1_{ \{T^k_n \leq t \} }=  \left(N^{k,X^k}_{T^k_n} - N^{k,X^k}_{T^k_{n-1}}\right) 1\!\!1_{ \{T^k_n \leq t \} }; n\ge 1.$$
Therefore, $\lim_{k\rightarrow \infty}W^k = 0$ weakly in $\textbf{B}^2$ over $[0,T]$ if, and only if,

\begin{equation}\label{wkak}
[W^k,A^k]_t=\sum_{n=1}^\infty \left(N^{k,X^k}_{T^k_n} - N^{k,X^k}_{T^k_{n-1}}\right) \Delta A^k_{T^k_n} 1\!\!1_{\{T^k_n \leq t  \} } \rightarrow 0
\end{equation}
weakly in $L^1(\mathbb{P})$ as $k\rightarrow \infty$, for every $t\in [0,T]$.
\end{lemma}
\begin{proof}
By applying Lemma 3.1 in \cite{LEAO_OHASHI2013} and the fact that $X\in \textbf{H}^2$, we have the bound $\sup_{k\ge 1}\mathbb{E}[\delta^kX,\delta^kX]_T < \infty$. Burkholder-Davis-Gundy inequality also yields $\sup_{k\ge 1}\mathbb{E}[X^k,X^k|_T < \infty$ and hence, $\sup_{k\ge 1} \mathbb{E} [W^k , W^k]_T < \infty.$ For a given $t\in (0,T]$, we have
\begin{eqnarray}
\nonumber \Delta W^k_{T^k_n} 1\!\!1_{ \{T^k_n \leq t \} } &=& \left(\Delta X^k_{T^k_n} - \Delta M^{k,X}_{T^k_n}\right) 1\!\!1_{ \{T^k_n \leq t \} }\\
\nonumber &=& \left( X^k_{T^k_n} - X^k_{T^k_n-} - \delta^k X_{T^k_n} + \delta^k X_{T^k_{n-1}}\right) 1\!\!1_{ \{T^k_n \leq t \} } \\
&=&\label{por1} \left( X^k_{T^k_n} - X^k_{T^k_n-} - X^k _{T^k_n} +  X^k_{T^k_{n-1}}\right) 1\!\!1_{ \{T^k_n \leq t \} }\\
\nonumber&=& \left(-X^k_{T^k_n-} + X^k_{T^k_{n-1}}\right) 1\!\!1_{ \{T^k_n \leq t \} } \\
&=&\label{por2} \left(N^{k,X^k}_{T^k_n} - N^{k,X^k}_{T^k_{n-1}}\right) 1\!\!1_{ \{T^k_n \leq t \} }, \quad n \geq 1,
\end{eqnarray}
where in (\ref{por1}) and (\ref{por2}), we have used identity (\ref{ident1}) and the fact that $N^{k,X^k}$ has continuous paths, respectively.
The last statement is a simple application of Lemmas \ref{energylemma}, \ref{lfund2} and the predictable martingale representation of the Brownian motion.
\end{proof}
We are now able to prove Lemma \ref{mainlemma}.

\

\noindent \textbf{Proof of Lemma \ref{mainlemma}:} Lemma \ref{carc_W} and predictability of $N^{k,X^k}$ yield $\Delta W^k_{T^k_n}  1\!\!1_{\{T^k_n \le t\}}$ is $\mathcal{G}^k_{n-}$-measurable for each $n\ge 1$ and $t\ge 0$. Therefore, it follows from Corollaries \ref{cormart} and \ref{AKMAR} that
$$\mathbb{E}[\Delta W^k_{T^k_n}\Delta A^k_{T^k_n}|\mathcal{G}^k_{n-}]1\!\!1_{\{T^k_n \le t\}} = \Delta W^k_{T^k_n}\mathbb{E}[\Delta A^k_{T^k_n}|\mathcal{G}^k_{n-}]1\!\!1_{\{T^k_n \le t\}}=0~a.s$$
for each $n\ge 1$ and $t\ge 0$. By applying Lemma \ref{martlemma} on the pure jump process $[W^k,A^k]$ given by (\ref{wkak}), we can safely state that this process is an $\mathbb{F}^k$-martingale for every $k\ge 1$. Lemma \ref{carc_W} yields

\begin{eqnarray*}
\sup_{k\ge 1} \mathbb{E} [W^k , W^k]_T =\sup_{k\ge 1} \mathbb{E} \sum_{n=1}^\infty  \left(N^{k,X^k}_{T^k_n} - N^{k,X^k}_{T^k_{n-1}}\right)^2 1\!\!1_{\{T^k_n \leq T\}} < \infty,
\end{eqnarray*}
so that
\begin{eqnarray*}
\mathbb{E} \Big[[W^k , A^k] , [W^k , A^k]\Big]_T &=& \mathbb{E} \sum_{n=1}^\infty \left(N^{k,X^k}_{T^k_n} - N^{k,X^k}_{T^k_{n-1}}\right)^2 |\Delta A^k_{T^k_n}|^2 1\!\!1_{\{T^k_n \leq t  \} } \\
&=&2^{-2k} \mathbb{E}[W^k,W^k]_T\le 2^{-2k}\sup_{r\ge 1}\mathbb{E}[W^r,W^r]_T\rightarrow 0
\end{eqnarray*}
as $k\rightarrow \infty$. Therefore, $\lim_{k\rightarrow \infty}[W^k , A^k] = 0$ strongly in $\textbf{B}^2$ over $[0,T]$ so that Lemma \ref{carc_W} yields $\lim_{k\rightarrow \infty}W^k=\big(X^k - M^{k,X}\big)=0$ weakly in $\textbf{B}^2$ over $[0,T]$. The set $\{M^{k,X}; k\ge 1\}$ is $\textbf{B}^2$-weakly relatively sequentially compact where all limits points are square-integrable $\mathbb{F}$-martingales over $[0,T]$. The weak convergence $\lim_{k\rightarrow \infty}\mathbb{F}^k=\mathbb{F}$ (see Lemma 2.2 in \cite{LEAO_OHASHI2013}) yields $\lim_{k\rightarrow \infty}X^k=X$ strongly in $\textbf{B}^1$. This allows us to conclude $\lim_{k\rightarrow \infty}M^{k,X}=X$ weakly in $\textbf{B}^2$. As a consequence, $\langle X, B\rangle^\delta_t = \lim_{k\rightarrow \infty}[M^{k,X},A^k]_t = [X,B]_t$ weakly in $L^1(\mathbb{P})$ for each $t\in [0,T]$.


\section{The new proofs of Theorem 3.1 and Corollary 4.1 in \cite{LEAO_OHASHI2013}}
\subsection{New proof of Theorem 3.1 in \cite{LEAO_OHASHI2013}} Let us define $N^X:=X - X_0 - M^X$. We claim that $\langle N^X, B \rangle^\delta=0 $. Indeed, $[\delta^k N^X, A^k] = [M^{k,X} - \delta^k M^X, A^k]$. Proposition 3.2 in \cite{LEAO_OHASHI2013} yields $[M^{k,X},A^k]_t\rightarrow [M^X,B]_t$ weakly in $L^1(\mathbb{P})$ for each $t\in [0,T]$. By noticing that $[\delta^kM^X, A^k] = [M^{k,M^X},A^k]_t; 0\le t\le T$, we shall apply Lemma \ref{mainlemma} to state that $\lim_{k\rightarrow \infty}[\delta^kM^X, A^k]_t=[M^X,B]_t$ weakly in $L^1(\mathbb{P})$ for every $t\in [0,T]$. Hence, $\langle N^X, B \rangle^\delta=0$. The uniqueness of the decomposition is now just a simple consequence of the martingale representation of the Brownian motion.

\subsection{New proof of Corollary 4.1 in \cite{LEAO_OHASHI2013}} In one hand, Lemma \ref{mainlemma} yields $\langle X, B\rangle^\delta = [X,B]$ for every $X\in \textbf{H}^2$. On the other hand, Theorem 4.1 in \cite{LEAO_OHASHI2013} yields $X_t = \int_0^t\mathcal{D}X_sdB_s; 0\le t\le T$. Representation (4.9) in \cite{LEAO_OHASHI2013} is then a simple consequence of the definition of $\mathcal{D}^k X$.

\section{Final remarks on Lemma 3.4 and 4.1 in \cite{LEAO_OHASHI2013}}
There are also minor modifications in the proofs of Lemma 3.4 and 4.1 in \cite{LEAO_OHASHI2013} due to the false statement written in Remark 2.2.

\textbf{Lemma 3.4 in \cite{LEAO_OHASHI2013}}: In the proof of Lemma 3.4 in \cite{LEAO_OHASHI2013}, there is a bad argument just below (3.9) in \cite{LEAO_OHASHI2013}. We wrote $\sup_{0\le t\le T}|H^k_t - H^k_{t-}| = \max_{n\ge 1}|H^k_{T^k_n} - H^k_{T^k_{n-1}}| 1\!\!1_{\{T^k_n \le T\}}$, which is not true due to Corollary~\ref{cormart}. However, the new argument is very simple:

\begin{eqnarray*}
|H^k_t- H^k_{t-\epsilon}|&\le& |H^k_t - H_t| + |H_t - H_{t-\epsilon}| + |H_{t-\epsilon} - H^k_{t-\epsilon}|\\
& &\\
&\le& 2\sup_{0\le u\le T}|H^k_u - H_u| + |H_t - H_{t-\epsilon}|~a.s,
\end{eqnarray*}
for $0\le t\le T$ and $\epsilon>0$. Therefore, $\sup_{0\le t\le T}|H^k_t - H^k_{t-}|\le 2\sup_{0\le u\le T}|H^k_u - H_u|\rightarrow 0$ in probability as $k\rightarrow\infty$ due to Lemma 2.2 (ii) in \cite{LEAO_OHASHI2013}.

\textbf{Proof of Lemma 4.1 in \cite{LEAO_OHASHI2013}}: Let $X_t = \mathbb{E}[g|\mathcal{F}_t]; 0\le t\le T$. By the very definition, we have $\delta^kX_{T^k_n} = \mathbb{E}[g|\mathcal{G}^k_n];n\ge 0$. Do the same splitting as in equation (4.2) in \cite{LEAO_OHASHI2013}. Since $X$ is bounded, we know that $\lim_{k\rightarrow \infty}\delta^kX = X$ strongly in $\textbf{B}^p$ for every $p\ge 1$, so that the first and last terms vanish in equation (4.2) in \cite{LEAO_OHASHI2013} . The second term in equation (4.2) in \cite{LEAO_OHASHI2013} also vanishes due to the path continuity of $X$ and the fact
$\lim_{k\rightarrow \infty}\mathbb{E}\max_{n\ge 1}|\Delta T^k_n| 1\!\!1_{\{T^k_n\le T\} }=0.$

\begin{remark}
Identity (\ref{ident1}) and Corollary \ref{cormart} imply that, in general, equation (3.4) in \cite{LEAO<OHASHI2013.1} only holds at the stopping times $(T^k_n)_{n\ge 0}$. Lemma 10 and Theorem 11 in \cite{LEAO<OHASHI2013.1} is then a simple consequence of Lemma \ref{mainlemma} in the multi-dimensional case.
\end{remark}


\begin{thebibliography}{15}



\bibitem{LEAO<OHASHI2013.1} Bonetti, D., Le\~ao, D., Ohashi, A. and Siqueira, V. (2015). A general multidimensional Monte Carlo approach for dynamic hedging under stochastic volatility. \textit{Int. J. Stoch. Anal}, \textbf{v} 2015.

\bibitem{lejan} Le Jan, Y. (1978). Temps d'arret stricts et martingales de sauts. \textit{Z. Wahrscheinlichkeitstheorie verw. Gebiete},\textbf{
44}, 213-225.

\bibitem{jacod} Jacod, J., and Skohorod. A.V. (1994). Jumping
filtrations and martingales with finite variation. \textit{Lecture Notes in Math.} \textbf{1583}, 21-35. Springer.

\bibitem{LEAO_OHASHI2013} Le\~{a}o, D. and Ohashi, A. (2013). Weak approximations for Wiener functionals. \textit{Ann. Appl. Probab,} 23, \textbf{4}, 1660-1691.

\end{thebibliography}
\end{document}